\documentclass[smallextended,referee,envcountsect]{svjour3}
\smartqed
\usepackage{graphicx}

\usepackage{amsbsy,amsfonts,amsmath,amssymb,bbm,calc,caption,color,dsfont,graphics}
\usepackage{graphicx,ifthen,latexsym,mathptmx,mathrsfs,multicol}
\usepackage{overpic,pifont,psfrag,rotating,stmaryrd,theorem}
\usepackage{booktabs}
\spnewtheorem{algorithm}{Algorithm}{\bf}{\it}

\journalname{JOTA}

\begin{document}

\title{A New Modified Newton-Type Iteration Method for Solving Generalized Absolute Value Equations}


\titlerunning{SSMN iteration method for GAVEs}        

\author{Xu Li \and Xiao-Xia Yin }


\institute{Xu Li \at
	Department of Applied Mathematics, Lanzhou University of Technology, Lanzhou 730050, PR China \\
	\email{lixu@lut.edu.cn}             \\
	\and
	Xiao-Xia Yin \at
	Department of Applied Mathematics, Lanzhou University of Technology, Lanzhou 730050, PR China \\
	\email{2352611979@qq.com}
}

\date{Received: date / Accepted: date}

\maketitle

\begin{abstract}
A shift splitting modified Newton-type (SSMN) iteration method is introduced for solving large sparse generalized absolute value equations (GAVEs). The SSMN method is established by replacing the regularized splitting of the coefficient matrix of the linear part, which is employed in the modified Newton-type (MN) iteration method, with the shift splitting of the matrix. The conditions for the convergence of the proposed method are discussed in depth for the cases when the coefficient matrix is a general matrix, a symmetric positive definite matrix, and an $H_{+}$-matrix. Through two numerical examples, we find that the SSMN and MN methods complement each other. The optimal performances of the two methods depend on the definiteness of the coefficient matrix of the linear part. The MN method is more efficient when the coefficient matrix is positive definite, whereas the SSMN method has a better performance when the coefficient matrix is indefinite.
\end{abstract}



\section{Introduction}
The generalized absolute value equation (GAVE) is formulated as follows:
\begin{equation}\label{1.1}
	Ax-B|x|=b,
\end{equation}
where $A, \ B\in \mathbb{R}^{n\times n}, b\in \mathbb{R}^{n}$ are given, and $|x|=(|x_1|, \ldots, |x_n|)^T\in \mathbb{R}^{n}$ denotes the componentwise absolute value of the unknown $x$. If  $B = I $, where $I $ describes the identity matrix, the GAVE \eqref{1.1} can be simplified to the following absolute value equation (AVE)
\begin{equation}\label{1.2}
	Ax-|x|=b.
\end{equation}
If $B$ is a nonsingular matrix, the GAVE \eqref{1.1} can be equivalently reformulated as the AVE \eqref{1.2}. Due to the nonlinear property of $|x|$, the GAVE \eqref{1.1} is considered as the following nonlinear system
\begin{equation}\label{1.3}
	F(x)=0, ~\text{with}~ F(x)=Ax-B|x|-b.
\end{equation}

The GAVEs have arisen in various scientific computing fields and engineering applications since they were first introduced by Rohn \cite{AVE1}. One of important research background is the following linear complementarity problem (LCP) \cite{LCP}:
to derive two real vectors $z, \omega \in \mathbb{R}^{n}$ such that
\begin{equation}\label{1.4}
	z\geq 0,\quad\omega=Mz+q\geq 0,\quad z^T\omega=0,
\end{equation}
where $M\in \mathbb{R}^{n\times n}$ and $q\in \mathbb{R}^{n}$ are given. From \cite{AVE4,AVE5,AVE6}, the LCP($q$, $M$) \eqref{1.4} can be formulated as the following GAVE:
\begin{equation}
	(M+I)x-(M-I)|x|=q,
\end{equation}
with
\begin{equation*}
	x=\frac{1}{2}((M-I)z+q).
\end{equation*}
Besides the well-known LCPs, other frequent optimization problems can  also be transformed into the GAVEs \eqref{1.1}, like linear programming and convex quadratic programming \cite{AVE1,AVE2}.

In the past two decades, the GAVE and AVE have attracted increasingly interests of various researchers due to its simple and special structure. In the aspect of theoretical analysis, Mangasarian \cite{AVE5,AVE2} proved that the AVE is NP-hard and gave some existence results of solutions for GAVE and AVE. Subsequently, some scholars \cite{AVE9,unique-Wu-1,unique-Wu-2,unique-Wu-3,unique-Wu-4} further improved the solvability and unique solution theories of GAVE and AVE.  In the aspect of numerical algorithms, iterative methods are the most used method. The early iterative methods include the finite succession of linear programs \cite{AVE13}, the optimization method \cite{AVE17}, the sign accord method \cite{AVE15}, the nonlinear HSS-like method \cite{HSS-like,PHSS-like} and the SOR-like method \cite{SOR-like-1,SOR-like-2}. In order to apply the iteration methods of nonlinear systems to the GAVE, Mangasarian originally introduced the subgradient for the nonlinear term $|x|$ to overcome the non-differentiability, and presented the generalized Newton (GN)  iterative method \cite{AVE19}. Thereafter, some scholars established some more efficient methods based on the GN method, for example, the GN method with stable quadratic convergence \cite{AVE24}, the generalized Traub method \cite{AVE25}, the modified GN method \cite{AVE26} and the relaxed GN method \cite{relaxed-GN}. However, those methods require expensive computing costs in actual computations because the coefficient matrix of each iteration is changing. Then Rohn et al. proposed the following more practical Picard iteration method \cite{AVE11,AVE16} to solve GAVE \eqref{1.1}:
\begin{equation}\label{1.5}
	Ax^{(k+1)}=B|x^{(k)}|+b,\quad k=0,1,2,\ldots.
\end{equation}

Recently, for the nondifferentiable system of nonlinear equations in Banach space, Han \cite{AVE28} established a modified Newton-type (MN) iteration method through the separation of differential and non-differential parts. Lately, inspired by \cite{AVE28}, Wang et al. \cite{WangCao2019} presented the MN iteration method for the GAVE \eqref{1.1} by choosing a special differential part. The Picard iteration \eqref{1.5} is a special case of the MN iteration. In this paper, applying the same idea of \cite{AVE28} and utilizing the shift splitting \cite{SS} of coefficient matrix $A$, we propose a shift splitting modified Newton-type (SSMN) iteration method for the GAVE \eqref{1.1}. Compared with the MN method, the iteration scheme of our method is more balanced. Numerical experiments show the efficiency of the SSMN method.

The organization of the remaining parts of the current work is illustrated in the following. In Section 2, we present a brief introduction of method in \cite{AVE28,WangCao2019} and establish the  SSMN method for solving the GAVE \eqref{1.1}. In Section 3, some convergence conditions for the SSMN iteration method are investigated in detail. We give two numerical examples in Section 4 to verify the effectiveness of the SSMN method. Finally, the conclusions are given in Section 5.

\section{A shift splitting Modified Newton-Type (SSMN) Iteration Method}
Firstly, we consider the more general nonlinear system
\begin{equation}\label{2.1}
	F(x)=0, \quad\text{with}\quad F(x)=H(x)+G(x),
\end{equation}
where $H(x)$ and $G(x)$ are differentiable and Lipschitz continuous functions, respectively. Han \cite{AVE28} established the following modified Newton-type (MN) iteration method for finding the solution of the equation \eqref{2.1}:
\begin{equation}
	H'(x^{(k)})x^{(k+1)}=H'(x^{(k)})x^{(k)}-(H(x^{(k)})+G(x^{(k)})),\quad k=0,1,2,\ldots,
\end{equation}
where the Jacobian matrix $H'(x^{(k)})$ is nonsingular.

Then for the equivalent nonlinear system \eqref{1.3} of the GAVE \eqref{1.1}, different iteration methods can be obtained by choosing appropriate splitting of $F(x)$. By using the following regularized splitting for the matrix $A$,
\begin{equation*}
	A=(\Omega+A)-\Omega,
\end{equation*}
where $\Omega\in \mathbb{R}^{n\times n}$ indicates a positive semi-definite (PSD) matrix,
Wang et al. introduced the following way to choose $H(x)$ and $G(x)$:
\begin{equation*}
	H(x)=(\Omega+A)x \quad\text{and}\quad G(x)=-\Omega x-B|x|-b,
\end{equation*}
and proposed the MN iterative scheme to solve the GAVE \eqref{1.1}:
\begin{equation}\label{2.2}
	(\Omega+A)x^{(k+1)}=\Omega x^{(k)}+B|x^{(k)}|+b,\quad k=0,1,2,\ldots.
\end{equation}

By comparing the coefficient matrix at each iteration step, we can find that the MN iteration has a better-conditioned coefficient matrix than the Picard iteration \eqref{1.5} considering the PSD matrix $\Omega$. In particular, the MN iterative method \eqref{2.2} is reduced to the Picard iterative method \eqref{1.5} when $\Omega=\textbf{0}$, with $\textbf{0}$ standing for the zero matrix.

In this paper, in order to improve the computational efficiency, we apply another splitting for the matrix $A$
\begin{equation*}
	A=\frac{1}{2}(\Omega+A)-\frac{1}{2}(\Omega-A),
\end{equation*}
which can be regarded as a generalized shift-splitting \cite{SS}. Then we can obtain another way to choose $H(x)$ and $G(x)$:
\begin{equation*}
	H(x)=\frac{1}{2}(\Omega+A)x \quad\text{and}\quad G(x)=-\frac{1}{2}(\Omega-A) x-B|x|-b,
\end{equation*}
and establish a shift splitting modified Newton-type (SSMN) iterative scheme to solve the GAVE \eqref{1.1}:
\begin{equation}\label{2.3}
	(\Omega+A)x^{(k+1)}=(\Omega-A)x^{(k)}+2(B|x^{(k)}|+b),\quad k=0,1,2,\ldots.
\end{equation}

In comparison with the MN iteration, the coefficient matrix is unchanged, so the SSMN iteration can preserve the advantages of the MN iteration.  Moreover, from \cite{SS} we know that, the shift-splitting has a better convergence behavior than regularized splitting because the former is more balanced.

\section{Convergence Analysis}

In the current section, some sufficient conditions are presented to ensure the convergence of the SSMN method to solve the GAVE \eqref{1.1}. With regard to the AVE \eqref{1.2}, the convergence conditions of the SSMN method can be immediately attained by taking $B=I$ from the following theorems.

We first give two convergence conditions in the general case, then some special convergence conditions are obtained when $A$ is a symmetric positive definite (SPD) matrix or an $H_{+}$-matrix.

\subsection{The General Case}
In the current subsection, the following theorems are presented to verify the convergence of the SSMN iteration approach when the matrices $\Omega+A$ and $A$ are nonsingular, respectively.

\begin{theorem}\label{thm1}
	Suppose that $A,B\in \mathbb{R}^{n\times n}$, and $\Omega \in \mathbb{R}^{n\times n}$ is PSD such that  $\Omega+A$ is nonsingular. If
	\begin{equation}\label{3.1}
		||(\Omega+A)^{-1}||_2<\frac{1}{||\Omega-A||_2+2||B||_2},
	\end{equation}
	then the iteration sequence $\{x^{(k)}\}_{k=0}^{+\infty}$ generated by the SSMN iteration method converges to the solution $x^{*}$ of the GAVE \eqref{1.1} for any initial vector.
\end{theorem}
\begin{proof}
	Since $x^{*}$ is the solution of the GAVE \eqref{1.1}, it holds that
	\begin{equation}\label{3.2}
		(\Omega+A)x^{*}=(\Omega-A)x^{*}+2(B|x^{*}|+b).
	\end{equation}
	Subtracting \eqref{3.2} from \eqref{2.3}, we obtain
	\begin{equation}
		(\Omega+A)(x^{(k+1)}-x^{*})=(\Omega-A)(x^{(k)}-x^{*})+2B(|x^{(k)}|-|x^{*}|).
	\end{equation}
	According to the nonsingularity of matrix $A+\Omega$, we have
	\begin{equation}\label{3.3}
		x^{(k+1)}-x^{*}=(\Omega+A)^{-1}\big((\Omega-A)(x^{(k)}-x^{*})+2B(|x^{(k)}|-|x^{*}|)\big).
	\end{equation}
	Taking the 2-norm on both sides of \eqref{3.3}, we have
	\begin{equation}
		||x^{(k+1)}-x^{*}||_2\leq||(\Omega+A)^{-1}||_2(||\Omega-A||_2+2||B||_2)||x^{(k)}-x^{*}||_2.
	\end{equation}
	Then from condition \eqref{3.1} we draw the conclusion.
\end{proof}

\begin{theorem}\label{thm2}
	Suppose that $A\in \mathbb{R}^{n\times n}$ is nonsingular, $B\in \mathbb{R}^{n\times n}$, and $\Omega \in \mathbb{R}^{n\times n}$ is PSD such that $A+\Omega$ is nonsingular. If
	\begin{equation}\label{3.4}
		||A^{-1}||_2<\frac{1}{||\Omega-A||_2+||\Omega||_2+2||B||_2},
	\end{equation}
	then the iteration sequence $\{x^{(k)}\}_{k=0}^{+\infty}$ generated by the SSMN iteration method converges to the solution $x^{*}$ of the GAVE \eqref{1.1} for any initial vector.
\end{theorem}
\begin{proof}
	Based on  the Banach perturbation Lemma \cite{Golub1996}, the following estimation can be attained under the condition \eqref{3.4}
	\begin{equation*}
		\begin{split}
			||(\Omega+A)^{-1}||_2
			\displaystyle&\leq\frac{||A^{-1}||_2}{1-||A^{-1}||_2\cdot||\Omega||_2}\\
			\displaystyle&<\frac{\frac{1}{||\Omega-A||_2+||\Omega||_2+2||B||_2}}{1-\frac{||\Omega||_2}{||\Omega-A||_2+||\Omega||_2+2||B||_2}}\\
			\displaystyle&=\frac{1}{||\Omega-A||_2+2||B||_2}.
		\end{split}
	\end{equation*}
	Then from Theorem \ref{thm1}, the conclusion can be obtained.\\
\end{proof}


\subsection{The SPD Matrix Case}
In the current subsection, the convergence conditions of the SSMN iteration method \eqref{2.3} are obtained when $A$  and $\Omega=\omega I$ are SPD and positive scalar matrices, respectively.

\begin{theorem}
	Suppose that  $A\in \mathbb{R}^{n\times n}$ is a SPD matrix, $\Omega=\omega I\in \mathbb{R}^{n\times n}$ with $\omega>0$.
	Let the minimum and maximum eigenvalues of matrix A are indicated by $\mu_\text{min}$ and $\mu_\text{max}$, respectively. Define $\tau=||B||_2$.
	Then the iteration sequence $\{x^{(k)}\}_{k=0}^{+\infty}$ generated by the SSMN iteration method converges to the solution $x^{*}$ of the GAVE \eqref{1.1} for any initial vector, provided that the following conditions hold
	\begin{equation}\label{3.6}
		\tau<\mu_\text{min}\quad \text{and} \quad \omega>\frac{\mu_\text{max}-\mu_\text{min}}{2}+\tau.
	\end{equation}
	
\end{theorem}
\begin{proof}
	As $\Omega=\omega I$, we have
	\begin{equation*}
		\begin{split}
			||\Omega-A||_2
			\displaystyle&=\max\limits_{\mu\in \text{sp}(A)}|\omega-\mu|\\
			\displaystyle&=\max\{|\omega-\mu_\text{max}|,|\omega-\mu_\text{min}|\}\\
			\displaystyle&=\left\{
			\begin{array}{ll}
				\mu_\text{max}-\omega, \quad \text{for}\quad  \omega\leq\frac{\mu_\text{max}+\mu_\text{min}}{2}\\
				\omega-\mu_\text{min}, \quad \text{for}\quad  \omega\geq\frac{\mu_\text{max}+\mu_\text{min}}{2}
			\end{array}
			\right.
		\end{split}
	\end{equation*}
	Hence,
	\begin{equation*}
		\begin{split}
			||(A+\Omega)^{-1}||_2(||\Omega-A||_2+2||B||_2)
			\displaystyle&=\left\{
			\begin{array}{ll}
				\frac{\mu_\text{max}-\omega+2\tau}{\omega+\mu_\text{min}}, \quad \text{for}\quad  \omega\leq\frac{\mu_\text{max}+\mu_\text{min}}{2}\\
				\frac {\omega-\mu_\text{min}+2\tau}{\omega+\mu_\text{min}}, \quad \text{for}\quad  \omega\geq\frac{\mu_\text{max}+\mu_\text{min}}{2}
			\end{array}
			\right.
		\end{split}
	\end{equation*}
	Through explicit solving of the following inequalities
	\begin{equation*}
		\frac{\mu_\text{max}-\omega+2\tau}{\omega+\mu_\text{min}}<1
	\end{equation*}
	when $\omega\leq\frac{\mu_\text{max}+\mu_\text{min}}{2}$ and
	\begin{equation*}
		\frac{\omega-\mu_\text{min}+2\tau}{\omega+\mu_\text{min}}<1
	\end{equation*}
	when $\omega\geq\frac{\mu_\text{max}+\mu_\text{min}}{2}$, since the upper bound should be larger than the lower bound, the convergence conditions \eqref{3.6} can be derived for the iteration parameter $\omega$.
\end{proof}


\subsection{The $H_{+}$-matrix Case}
Firstly, some required notations and definitions are presented; see \cite{LCP,AVE6,WangCao2019}. Let $V=(v_{ij})\in \mathbb{R}^{n\times n}$. We use $|V|=(|v_{ij}|)\in \mathbb{R}^{n\times n}$ to represent the absolute value of the matrix $V$. The matrix $V$ is called a $Z$-matrix if each of its off-diagonal entries is non-positive. The nonsingular $Z$-matrix $V$  is called an $M$-matrix if $V^{-1}$ is a nonnegative matrix. The matrix $V$ is called an $H$-matrix when its corresponding comparison matrix $\langle V \rangle=(\langle v \rangle _{ij})\in \mathbb{R}^{n\times n}$ is an $M$-matrix, where
\begin{equation*}
	\langle v_{ij}\rangle=
	\left\{
	\begin{aligned}
		|v_{ij}|,  && i=j,\\
		-|v_{ij}|, && i\neq j,
	\end{aligned}
	\qquad i,j=1,2,\cdots,n.
	\right.
\end{equation*}
In particular, an $H$-matrix is called an $H_{+}$-matrix if each of its diagonal entries is positive.

The following theorem gives the condition for the convergence of the SSMN iteration method \eqref{2.3} when $A$ and $\Omega$ are an $H_{+}$-matrix and a positive scalar matrix, respectively.

\begin{theorem}
	Suppose that $A\in \mathbb{R}^{n\times n}$ is an $H_{+}$-matrix, $B\in \mathbb{R}^{n\times n}$, and $\Omega \in \mathbb{R}^{n\times n}$ is a positive scalar matrix. When
	\begin{equation}\label{3.5}
		||(\Omega+\langle A\rangle)^{-1}||_2<\frac{1}{||{\Omega}+|A|+2|B|~||_2},
	\end{equation}
	then the iteration sequence $\{x^{(k)}\}_{k=0}^{+\infty}$ generated by the SSMN iteration method converges to the solution $x^{*}$ of the GAVE \eqref{1.1} for any initial vector.
\end{theorem}
\begin{proof}
	By applying the absolute values on both sides of \eqref{3.3}, we have
	\begin{equation*}
		\begin{split}
			|x^{(k+1)}-x^{*}|
			\displaystyle&\leq|(\Omega+A)^{-1}|(|\Omega-A||x^{(k)}-x^{*}|+2|B||x^{(k)}-x^{*}|)\\
			\displaystyle&\leq(\Omega+\langle A\rangle)^{-1}(\Omega+|A|+2|B|)|x^{(k)}-x^{*}|\\
			\displaystyle&:=Z|x^{(k)}-x^{*}|,
		\end{split}
	\end{equation*}
	where we use the following estimate from \cite{AVE30}
	\begin{equation}
		|(\Omega+A)^{-1}|\leq(\Omega+\langle A\rangle)^{-1}.
	\end{equation}
	
	The iteration sequence $\{x^{(k)}\}_{k=0}^{+\infty}$ can converge to the solution $x^{*}$ when $\rho(Z)<1$, with $\rho(Z)$ denoting the spectral radius of the matrix $Z$.
	By the condition \eqref{3.5}, we have
	\begin{equation*}
		\begin{split}
			\rho(Z)
			\displaystyle&\leq||(\Omega+\langle A\rangle)^{-1}(\Omega+|A|+2|B|)||_2\\
			\displaystyle&\leq||(\Omega+\langle A\rangle)^{-1}||_2\cdot||\Omega+|A|+2|B|\,||_2\\
			\displaystyle&<1,
		\end{split}
	\end{equation*}
	then the proof is completed.
\end{proof}

%

\section{Numerical Results}

This section presents two numcerial examples to compare the numerical performance of the SSMN and MN iteration methods in terms of the number of iteration steps (denoted by ``IT"), the elapsed CPU time in second (denoted by ``CPU") and the residual (represented by ``RES") which is described by
\begin{equation*}
	\mathrm {RES}(x^{(k)}):=\frac{||Ax^{(k)}-B|x^{(k)}|-b||_2}{||b||_2},
\end{equation*}
where $x^{(k)}$ indicates the $k$th approximate solution to the GAVE \eqref{1.1}. In \cite{WangCao2019}, numerical examples have shown that the MN method outperforms the generalized Newton iteration method \cite{AVE19}, the modified generalized Newton (MGN) iteration method \cite{AVE26} and the Picard iteration method \cite{AVE11}. So in this paper, we just compare our methods with the MN method. The numerical experiments are performed in Matlab on an Intel(R) Core(TM) i5-6300U processor (2.40GHz, 8GB RAM).

In our implementations,  the initial guess is the zero vector. All iterations are terminated once $x^{(k)}$ satisfies $\mathrm {RES}(x^{(k)})\leq 10^{-7}$ or the number of iterations exceeds from a predefined iteration number $k_\text{max}=5000$.

For convenience, we take $\Omega=\omega I$ in both MN and SSMN methods. Here, the optimal experimental value $\omega_\mathrm{exp}$ is utilized as the iteration parameter $\omega$ in MN and SSMN iteration methods to minimize the number of iteration steps.
In addition, the Cholesky and LU factorizations are utilized to solve all the subsystems when $\omega I+A$ are SPD and nonsymmetric, respectively.

\paragraph{Example 1} (See \cite{AVE6})
Consider the LCP($q$, $M$), where
$M\in \mathbb{R}^{n\times n}$ is defined as $M=\widehat{M}+\mu I\in \mathbb{R}^{n \times n}$, and $q\in \mathbb{R}^{n}$ is described by $q=-Mz^{*}\in \mathbb{R}^{n}$, where
\begin{equation*}
	\widehat{M}=\mathrm {Tridiag}(-I,S,-I)=\begin{bmatrix}
		S&-I&0&\cdots&0&0\\
		-I&S&-I&\cdots&0&0\\
		0&-I&S&\cdots&0&0\\
		\vdots&\vdots& &\ddots&\vdots&\vdots\\
		0&0&\cdots&\cdots&S&-I\\
		0&0&\cdots&\cdots&-I&S
	\end{bmatrix}\in \mathbb{R}^{n\times n}
\end{equation*}
is a block-tridiagonal matrix,
\begin{equation*}
	S=\mathrm {tridiag}(-1,4,-1)=\begin{bmatrix}
		4&-1&0&\cdots&0&0\\
		-1&4&-1&\cdots&0&0\\
		0&-1&4&\cdots&0&0\\
		\vdots&\vdots& &\ddots&\vdots&\vdots\\
		0&0&\cdots&\cdots&4&-1\\
		0&0&\cdots&\cdots&-1&4
	\end{bmatrix}\in \mathbb{R}^{m\times m}
\end{equation*}
is a tridiagonal matrix, $n=m^{2}$, and $z^{*}=(1,2,1,2,\ldots,1,2,\ldots)^{T}\in \mathbb{R}^{n}$ indicates the unique solution of the LCP($q$, $M$).

The optimal experimental parameters, the iteration steps, the CPU times, and the residuals derived by MN and SSMN iteration methods with $\mu=-4$, $\mu=-1$ and $\mu=4$ for various problem sizes $n$ are compared in Tables \ref{table-1}-\ref{table-3}, respectively.

\begin{table}
	\caption{Numerical results for Example 1 with $\mu=-4$}\label{table-1}
	\centering
	\begin{tabular}{llllllll}
		\hline
		\toprule
		&Method &$n$& $100^2$  & $200^2$  & $300^2$   & $400^2$& $500^2$ \\
		\hline
		&MN  &$\omega_\mathrm{exp}$   &4.3    &4.4    &4.3     &4.4&4.4 \\
        &&IT   &43    &42    &41      &41  &40\\
		&&CPU  &0.2629  &1.6487 &5.2247 & 11.6777&21.6347\\
        &&RES  & 8.29e-08  &8.75e-08  &8.06e-08 & 8.00e-08&9.24e-08\\
		&SSMN  &$\omega_\mathrm{exp}$   &6.1    &6.0    &6.1      &5.9&6.0 \\
        &&IT   &24    &23      &22   &22 &22\\
		&&CPU  &0.1372  &1.0304  &2.9610 &7.1242&14.7887\\
        &&RES  &6.25e-08  &6.24e-08  &9.38e-08 &6.26e-08& 6.39e-08\\		
		\bottomrule
		\hline
	\end{tabular}
\end{table}

\begin{table}
	\caption{Numerical results for Example 1 with $\mu=-1$}\label{table-2}
	\centering
	\begin{tabular}{llllllll}
		\hline
		\toprule
		&Method &$n$& $100^2$  & $200^2$  & $300^2$   & $400^2$& $500^2$ \\
		\hline
		&MN  &$\omega_\mathrm{exp}$   &1.9   &1.9   &1.9      &1.9 &1.9\\
        &&IT   &41    &41    &40     &40 &40\\
		&&CPU  &0.2139  &1.6057  &4.8538 &11.4176&21.6197\\
        &&RES  & 6.99e-08  &5.39e-08  &8.84e-08 & 8.01e-08 &7.45e-08\\
		&SSMN  &$\omega_\mathrm{exp}$   &5.9  &5.9    &5.6      &5.8&5.5\\
        &&IT   &31    &31      &30   &31 &29\\
		&&CPU  &0.1757  &1.4100  &4.0520 &11.6920 &23.1958\\
        &&RES  &6.84e-08  &7.21e-08  &7.07e-08 &5.99e-08 &9.90e-08\\			
		\bottomrule
		\hline
	\end{tabular}
\end{table}

\begin{table}
	\caption{Numerical results for Example 1 with $\mu=4$}\label{table-3}
	\centering
	\begin{tabular}{llllllll}
		\hline
		\toprule
		&Method &$n$& $100^2$  & $200^2$  & $300^2$   & $400^2$& $500^2$ \\
		\hline
		&MN  &$\omega_\mathrm{exp}$   &4.9    &5.3    &5.7      &5.5&5.0 \\
	&&IT   &11   &11   &12      &12 &11\\
	&&CPU  &0.0568 &0.4315 &1.6151 & 4.7597&8.3397\\
	&&RES  & 5.77e-08  &5.69e-08  &5.41e-08 & 2.71e-08&2.60e-08\\
	&SSMN  &$\omega_\mathrm{exp}$   &19.0    &18.9    &19.3      &19.3&19.1 \\
	&&IT   &16   &16     &15   &15 &15\\
	&&CPU  &0.0907&0.7240  &2.6777  &6.1575&11.3842\\
	&&RES  &7.09e-08  &5.98e-08 &9.81e-08 &9.47e-08& 8.62e-08\\			
		\bottomrule
		\hline
	\end{tabular}
\end{table}

According to the results of Tables \ref{table-1}-\ref{table-3}, both the MN and SSMN iteration methods can converge to the solution of the LCP($q$, $M$) in all the above cases. Notably, if $\mu=-4$, which means that the matrices $A$ and $M$ are symmetric indefinite, and $\mu=-1$, which means that matrix $A$ is SPD and matrix $M$ is symmetric indefinite, SSMN method provides superior results compared with the MN method in terms of both the IT and CPU times. The superiorities of SSMN method disappear for $\mu=4 $, which means that both the matrices $A$ and $M$ are SPD.

\paragraph{Example 2} (See \cite{AVE6})
Consider the LCP($q$, $M$), where
$M\in \mathbb{R}^{n\times n}$ is defined as $M=\widehat{M}+\mu I\in \mathbb{R}^{n \times n}$, and $q\in \mathbb{R}^{n}$ is described by $q=-Mz^{*}\in \mathbb{R}^{n}$, where
\begin{equation*}
	\widehat{M}=\mathrm {Tridiag}(-1.5I,S,-0.5I)=\begin{bmatrix}
		S&-0.5I&0&\cdots&0&0\\
		-1.5I&S&-0.5I&\cdots&0&0\\
		0&-1.5I&S&\cdots&0&0\\
		\vdots&\vdots& &\ddots&\vdots&\vdots\\
		0&0&\cdots&\cdots&S&-0.5I\\
		0&0&\cdots&\cdots&-1.5I&S
	\end{bmatrix}\in \mathbb{R}^{n\times n}
\end{equation*}
is a block-tridiagonal matrix,
\begin{equation*}
	S=\mathrm {tridiag}(-1.5,4,-0.5)=\begin{bmatrix}
		4&-0.5&0&\cdots&0&0\\
		-1.5&4&-0.5&\cdots&0&0\\
		0&-1.5&4&\cdots&0&0\\
		\vdots&\vdots& &\ddots&\vdots&\vdots\\
		0&0&\cdots&\cdots&4&-0.5\\
		0&0&\cdots&\cdots&-1.5&4
	\end{bmatrix}\in \mathbb{R}^{m\times m}
\end{equation*}
is a tridiagonal matrix, $n=m^{2}$, and $z^{*}=(1,2,1,2,\ldots,1,2,\ldots)^{T}\in \mathbb{R}^{n}$ indicates the unique solution of the LCP($q$, $M$).

In Tables \ref{table-4}-\ref{table-6}, we list the optimal experimental parameters, the iteration steps, the CPU times, and the residuals for the MN and SSMN methods for different problem sizes $n$ with $\mu=-4$, $\mu=-2$ and $\mu=4$, respectively.

\begin{table}
	\caption{Numerical results for Example 2 with $\mu=-4$}\label{table-4}
	\centering
	\begin{tabular}{llllllll}
		\hline
		\toprule
		&Method &$n$& $100^2$  & $200^2$  & $300^2$   & $400^2$& $500^2$ \\
		\hline
		&MN  &$\omega_\mathrm{exp}$   &4.3    &4.3   &4.4      &4.3&4.4 \\
		&&IT   &45   &43    &43      &42 &42\\
		&&CPU  &0.2329  &1.7009 &5.2207 & 12.6712&26.8596\\
		&&RES  & 8.00e-08  &9.40e-08  &8.79e-08 & 8.57e-08&8.74e-08\\
		&SSMN  &$\omega_\mathrm{exp}$   &5.8   &6.1    &5.9      &6.0&6.1\\
		&&IT   &24    &24      &23   &23 &23\\
		&&CPU  &0.1365  &1.0846  &3.1040 &8.1952&15.5925\\
		&&RES  &6.54e-08  &6.97e-08  &7.02e-08 &6.92e-08& 7.03e-08\\	
		\bottomrule
		\hline
	\end{tabular}
\end{table}

\begin{table}
	\caption{Numerical results for Example 2 with $\mu=-2$}\label{table-5}
	\centering
	\begin{tabular}{llllllll}
		\hline
		\toprule
        &Method &$n$& $100^2$  & $200^2$  & $300^2$   & $400^2$& $500^2$ \\
		\hline
		&MN  &$\omega_\mathrm{exp}$   &2.4  &2.3   &2.3      &2.3  &2.4\\
        &&IT   &48   &46  &45 &45&45\\
		&&CPU  &0.2460  &1.8003 &5.4037 &13.4490 &28.5645\\
        &&RES  & 8.94e-08  &8.96e-08  &9.42e-08 & 8.16e-08 &8.48e-08\\
		&SSMN  &$\omega_\mathrm{exp}$   &4.0   &3.9  &4.0      &4.1&3.8\\
        &&IT   &26  &25      &25  &25 &24\\
		&&CPU  &0.1468  &1.1207  &3.3344 &8.9065 &16.4175\\
        &&RES  &7.21e-08  &7.06e-08  &6.69e-08 &6.68e-08 &6.24e-08\\			
		\bottomrule
		\hline
	\end{tabular}
\end{table}

\begin{table}
	\caption{Numerical results for Example 2 with $\mu=4$}\label{table-6}
	\centering
	\begin{tabular}{llllllll}
		\hline
		\toprule
		&Method &$n$& $100^2$  & $200^2$  & $300^2$   & $400^2$& $500^2$ \\
		\hline
		&MN  &$\omega_\mathrm{exp}$   &5.4&5.1&5.3&5.4&4.8 \\
		&&IT   &11    &11   &11      &11 &11\\
		&&CPU  &0.0562  &0.4277 &1.6433 & 4.3514&7.9189\\
		&&RES  & 7.83e-08  &4.38e-08  &5.82e-08 & 8.12e-08&5.00e-08\\
		&SSMN  &$\omega_\mathrm{exp}$   &19.4    &20.4    &19.3    &20.1&19.4\\
		&&IT   &16  &16      &16  &16 &16\\
		&&CPU  &0.0905 &0.7723 &2.9442  &6.8240&11.4411\\
		&&RES  &7.61e-08 &9.89e-08  &5.59e-08 &7.34e-08& 4.54e-08\\		
		\bottomrule
		\hline
	\end{tabular}
\end{table}

The results in Tables \ref{table-4}-\ref{table-6} demonstrate the convergence of the two methods. Consistent with the results in Example 1, if $\mu=-4,-2$, which means that both the matrices $A$ and $M$ are unsymmetric indefinite, the SSMN method uses smaller iteration numbers and less CPU times than the MN method, if $\mu=4$ (both the matrices A and M are unsymmetric positive definite), the MN method outperforms the SSMN method.

Therefore, the SSMN iteration method is a powerful and efficient iterative approach to solve the LCP($q$, $M$), especially when at least one of $A$ and $M$ is indefinite.

\section{Conclusion}

In this paper, we have established a new iteration method, the SSMN method, for solving large sparse GAVEs. The SSMN method is based on the shift splitting of the coefficient matrix of the linear part, in contrast to the regularized splitting employed by the MN method. In comparison with the regularized splitting, the shift splitting is more balanced, enabling the SSMN method performing well for general coefficient matrices. Moreover, we have described convergence conditions for general and special coefficient matrices. Under some sufficient conditions for the coefficient matrices, the SSMN iteration method generates a sequence that converges to the exact solution for any initial vector. Furthermore, our numerical experiments have shown that the SSMN and MN methods complement each other. The optimal use of these two methods depends on the definiteness of the coefficient matrix of the linear part. On the one hand, when the coefficient matrix is positive definite, the MN method has a better performance. On the other hand, when the coefficient matrix is indefinite, the SSMN method is much more efficient. Therefore, if the definiteness of the coefficient matrix of the linear part is known, one can use this information to choose the suitable method for an optimal performance. If the definiteness of the matrix is unknown, we suggest one to use the SSMN method, which performs well for general coefficient matrices.

\section*{Acknowledgements}
This work was supported by the Natural Science Foundation of Gansu
Province (No. 20JR5RA464) and the National Natural Science Foundation of
China (No. 11501272).


\end{document}